\documentclass{amsart}

\usepackage[english]{babel}
\usepackage[utf8x]{inputenc}
\usepackage[T1]{fontenc}

\usepackage[top=3cm,bottom=3cm,left=3cm,right=3cm,marginparwidth=1.75cm]{geometry}
\usepackage{geometry}

\usepackage{amsmath}
\usepackage{amsthm}
\usepackage{amssymb}
\usepackage{graphicx}
\usepackage{mathrsfs}
\usepackage[colorinlistoftodos]{todonotes}
\usepackage[colorlinks=true, allcolors=black]{hyperref}
\usepackage{comment}

\theoremstyle{definition}

\theoremstyle{remark}
\newtheorem*{remark}{Remark}
\newtheorem{theorem}{Theorem}
\newtheorem{corollary}[theorem]{Corollary}
\newtheorem{lemma}[theorem]{Lemma}
\newtheorem{proposition}[theorem]{Proposition}
\newtheorem*{main}{Theorem}
\newtheorem*{mainc}{Corollary}

\begin{document}
\title{Measuring complexity of curves on surfaces}

\author{Max Neumann-Coto}
\address{Instituto de Matemáticas, Universidad Nacional Autónoma de México, Ciudad Universitaria, 04510 Mexico}
\email{max.neumann@im.unam.mx}

\author{Macarena Covadonga Robles Arenas}
\address{Facultad de Ciencias, Universidad Nacional Autónoma de México, Ciudad
Universitaria, 04510 Mexico}
\email{msigma@ciencias.unam.mx}


\date{July 18, 2018}

\maketitle

\begin{abstract}
We consider the relations between different measures of complexity for free homotopy classes of curves on a surface $\Sigma$, including the minimum number of self-intersections, the minimum length of the words representing them in a geometric presentation of $\pi_1(\Sigma)$, and the minimum degree of the coverings of $\Sigma$ to which they lift as embeddings.
\end{abstract}


\section*{Introduction}

In this paper we explore the relations between some `measures of complexity' for the free homotopy classes of immersed curves on a surface $\Sigma$. Among the many possible ways to measure complexity, we consider

\begin{itemize}
    \item The minimum number of self-intersections among the curves in the class $\alpha$, denoted by $i(\alpha)$. 
    \item The minimum length of a word representing $\alpha$ among all the `standard' presentations of $\pi_1(\Sigma)$, denoted by $l(\alpha)$.
    \item The minimum number of times that a  the class $\alpha$ crosses a fundamental region for $\Sigma$, for all the possible fundamental regions of $\Sigma$, denoted by $n(\alpha)$.
    \item The minimum number of crossings between $\alpha$ and all the essential curves in $\Sigma$, denoted by $m(\alpha)$.
\end{itemize}

If follows from a celebrated theorem of Scott \cite{S} that each immersed curve with minimal self-intersections on a surface $\Sigma$ lifts  as an embedded curve to some covering of $\Sigma$ of finite degree. Thus, it is natural to define

\begin{itemize}
    \item The minimum degree of a covering of $\Sigma$ to which $\alpha$ lifts as a closed embedded curve, denoted by $d(\alpha)$.
\end{itemize}

Another theorem, due to Hempel \cite{H}, and also a consequence of Scott's more general result, states that all surface groups are residually finite, i.e., that for each element of $\pi_1(\Sigma)$ there is a  a finite index subgroup that does not contain it, which is equivalent to saying that for each homotopy class of curves $\alpha$ in $\Sigma$ there is a finite covering to which $\alpha$ does not lift. Similarly to the previous measure, we define

\begin{itemize}
    \item The minimum degree of a covering of $\Sigma$ to which $\alpha$ does not lift, denoted by $r(\alpha)$. 
\end{itemize}

In \cite{P}, Patel obtained a bound for $d(\alpha)$ in terms of the length of some hyperbolic geodesic in the class 
$\alpha$ and in \cite{AGPS} that result was used to show that $d(\alpha) \leq c \cdot i(\alpha)$, for some $c$ that depends on the topology of $\Sigma$ (assuming that $\chi(\Sigma)<0$ and that $\alpha$ is nonsimple).
In \cite{R}, Rivin asked whether it is possible to bound $d(\alpha)$ in terms of $i(\alpha)$ alone. The main result of this paper, proved in chapter 2, is the following:

\begin{main}
For each homotopy class $\alpha$ on an orientable surface $\Sigma$, $d(\alpha) < 5(i(\alpha)+1)$.
\end{main}

As a by-product of the proof of the previous result we obtain:

\begin{mainc}
For each homotopy class $\alpha$ on an orientable surface $\Sigma$, $r(\alpha) \leq max \{3,  i(\alpha)\}$ if 
$\alpha \neq 0$ in $H_1(\Sigma)$, and  $r(\alpha) \leq 5(i(\alpha)+1)$ otherwise.
\end{mainc}

The \textit{geometric presentations} of $\pi_1(\Sigma)$ are those obtained by cutting $\Sigma$ along a graph to get a polygonal disc, and have one generator for each edge of the graph and one relation for each vertex in the interior of $\Sigma$. 
We define $n(\alpha)$ as the minimum length of the words in the conjugacy class of $\alpha$ among all these geometric presentations of $\pi_1(\Sigma)$. It is immediate that $n(\alpha)$ equals the minimum number of times that a curve in the class $\alpha$ crosses a fundamental region for $\Sigma$.
The \textit{standard presentations} of $\pi_1(\Sigma)$ are geometric presentations with a minimum number of generators and at most one relator. They arise from graphs that have just one vertex if $\Sigma$ is closed or consist of disjoint arcs if $\Sigma$ has boundary and are the most commonly studied presentation for $\pi_1(\Sigma)$, which is why it makes sense to define $l(\alpha)$ as the minimum word length in each class $\alpha$ taken among these presentations.

\begin{main} 
If $\alpha$ is a free homotopy class of curves on a surface $\Sigma$ then $n(\alpha) \leq i(\alpha)+1$.
\end{main}

It turns out that $l(\alpha)$ can be much larger than $n(\alpha)$ and that it cannot be bounded in terms of $n(\alpha)$ or $i(\alpha)$ alone. This was one of the difficulties in proving the main result, but we show the following. 

\begin{main} 
If $\alpha$ is a free homotopy class of curves in a closed surface $\Sigma$ then $l(\alpha)\leq (n(\alpha)+2) \cdot 2g(\Sigma)$.
\end{main} 

The ideas of the proof allow us to give a bound for $m(\alpha)$, which corresponds to the minimum number of times that that a curve in a homotopy class $\alpha$ crosses some simple essential curve on $\Sigma$.

\begin{main} 
If $\alpha$ is a free homotopy class of curves on a compact orientable surface $\Sigma$ then
$m(\alpha) \leq 2\sqrt{i(\alpha)}+1$.
\end{main}

One might expect a much better bound for $m(\alpha)$ in terms of $i(\alpha)$, but we give examples showing that 
$m(\alpha)$ can grow at least as $\frac{\sqrt{3}}{2}\sqrt{i(\alpha)}$.
As for other relations between the aforementioned complexity measures,
it is not hard to show that $i(\alpha) \leq \frac{1}{2}n(\alpha)(n(\alpha)-1)$ and that this bound is sharp.
Since there are curves with any number of self-intersections that lift to embeddings in coverings of degree 2, $i(\alpha)$ is not bounded by any function of $d(\alpha)$ alone. It is also clear that there is no bound for $n(\alpha)$ in terms of $m(\alpha)$.

\begin{remark}
The same measures of complexity can be used for homotopy classes of immersed arcs on a surface with boundary
(either fixing their endpoints or leaving them free to move along $\partial \Sigma$).
\end{remark}

\medskip

Some questions remain.
Many of the bounds given here are sharp, but the examples suggest that the bounds for $d(\alpha)$ and  $r(\alpha)$ in terms of $n(\alpha)$ or $i(\alpha)$  can be improved. Is it true that $d(\alpha) \leq 2n(\alpha)$?
Is there a better bound for $r(\alpha)$?
It also seems natural to ask whether or not these complexity measures behave monotonously.
If $\alpha$ is a free homotopy class of curves in $\Sigma$, then the `subloops' of $\alpha$ (the classes $\alpha'$ corresponding  to the subloops of a curve with minimal self-intersections in $\alpha$) depend only on $\alpha$. It follows  that $n(\alpha')< n(\alpha)$, $l(\alpha') < l(\alpha)$ and $m(\alpha') \leq m(\alpha)$, but it is not true that $r(\alpha') \leq r(\alpha)$. Is it true that $d(\alpha') \leq d(\alpha)$?

\bigskip

{\sl About the general structure of the paper}.

\bigskip

Chapter 1 starts by showing that $n(\alpha)$, $l(\alpha)$ and $m(\alpha)$ can be obtained by considering only curves with minimal self-intersections in $\alpha$. Then we prove that $n(\alpha)-1 \leq i(\alpha) \leq n(\alpha)(n(\alpha)-1)/2$, and show that $l(\alpha) \leq 2n(\alpha)-2$ if $\Sigma$ is planar and $\alpha$ is non-simple.
Then we introduce the notion of a \textit{peeling} for a curve (or more generally a graph) on a surface. This notion allows us to bound $l(\alpha)$ in terms of $n(\alpha)$ and $\chi(\Sigma)$ and to get the bound for $m(\alpha)$ in terms of $i(\alpha)$.

\

Chapter 2 is mainly made up of Theorem \ref{dn} and it's proof. Although the main idea of the proof is simple (naively speaking, it follows the method of Lemma \ref{s}) the details are tricky, and it requires a careful analysis of the different cases to guarantee that our construction works in its most general setting:

We start from a graph that cuts $\Sigma$ into a disc and use it to find a hierarchy for $\Sigma$, i.e., a collection of curves that cut $\Sigma$ into a surface with boundary $\Sigma'$ and a collection of arcs that cut $\Sigma'$ into a disc. Finding a hierarchy where the arcs connect different boundary components of $\Sigma'$ (while controlling the number of intersections with $\alpha$) allows us to lift the segments of $\alpha$ in $\Sigma'$ to coverings of $\Sigma'$ that can be glued together to get a partial covering of $\Sigma$ (that this is always possible follows from lemma \ref{p}).  To guarantee that this partial covering can be assembled so that it yields a covering of $\Sigma$ where the entire curve lifts simply, we `double' it.
Finally, we prove the corollary to the main theorem and some other bounds in the case that $\Sigma$ is planar.


\vskip10pt

\section{Intersections and word length}

We can define the quantities $i(\alpha)$, $l(\alpha)$, $n(\alpha)$ and $m(\alpha)$ for fixed curves on a surface, and then define them for each homotopy class as the minimum for all the curves in that class.
In order prove the inequalities relating $i(\alpha)$ with the other quantities for a homotopy class $\alpha$ it would be enough to prove the inequalities for all the curves in the class, but the inequalities for fixed curves do not always hold, and we will need to use the following lemma.


\begin{lemma} If $\alpha$ is a free homotopy class of curves in $\Sigma$, then $l(\alpha)$, $n(\alpha)$ and $m(\alpha)$ are attained by curves with minimal self-intersections in $\alpha$.
\end{lemma}

\begin{proof}
It follows from \cite{HS} and \cite{N} that each essential curve on a surface $\Sigma$ can be cut and pasted to obtain a freely homotopic curve with minimal self-intersections, so the images of the two curves differ only in a small neighbourhood of some self-intersection points. So the number of self-intersections of the curve can be minimized without increasing its intersections with a given curve or graph on $\Sigma$.
\end{proof}

Let's start with the relation between $n(\alpha)$ and $i(\alpha)$. It is known that if
$\alpha$ is a curve with minimal self-intersections on the surface $\Sigma$ then $i(\alpha) \leq \frac{1}{2}n(\alpha)(n(\alpha)-1)$.
For, if $n(\alpha)=n$ then $\Sigma$ can be cut into a disc $D$ that is crossed by $\alpha$ along $n$ arcs. As $\alpha$ has minimal self-intersections these arcs are embedded and intersect each other at most once, so the number of self-intersections of $\alpha$, which is the number of intersections between these arcs, is at most  $n(n-1)/2$.

\begin{theorem}{}  
\label{ni}
If $\alpha$ is an immersed curve on a compact surface $\Sigma$, then $n(\alpha) \leq i(\alpha)+1$. 
\end{theorem}

\begin{proof}{}
For the first part we need to show that if $\alpha$ is a curve with $n$ self-intersections then there exists a graph $G$ that cuts $\Sigma$ into a disc and intersects $\alpha$ no more than $n+1$ times.
Assume first that $\alpha$ {\sl fills} $\Sigma$, so the complementary regions are discs and boundary-parallel annuli.
Let $G'$ be the dual graph for $\alpha$ in $\Sigma$, that is, $G'$ has one vertex $p_i$ for each complementary region $R_i$ and one edge between $p_i$ and $p_j$ for each arc of $\alpha$ where $R_i$ meets $R_j$. 
So $G'$ intersects $\alpha$ once at each edge.
If $\Sigma$ has non-empty boundary we place the vertices of $G'$ corresponding to boundary regions in $\partial \Sigma$. 
Since $\alpha$ fills $\Sigma$, $G' \cup \alpha$ cuts $\Sigma$ into small quadrilaterals that have one corner in $p_i$, one corner in an intersection point of $\alpha$ and two corners in the middle points of edges of $G'$, so $G'$ cuts $\Sigma$ into quadrilaterals $C_i,C_2,...,C_n$ (each formed by the 4 small quadrilaterals around an intersection point of $\Sigma$).
We can glue $C_1,C_2,...,C_s$ along $s-1$ of its sides (which are edges of $G'$) to obtain a single polygonal disc $P$ that is a fundamental domain for $\Sigma$.
Since $P$ has $s-1$ edges of $G'$ in its interior, the boundary of $P$ has $2(n+1)$ sides (two for each remaining edge of $G'$). 
Let $G$ be the image of $\partial P$ in $\Sigma$. Then $G$ cuts $\Sigma$ into a disc and $G$ crosses $\alpha$ no more than $n+1$ times.

\begin{figure}[h!]
\centerline{\includegraphics[scale=0.8]{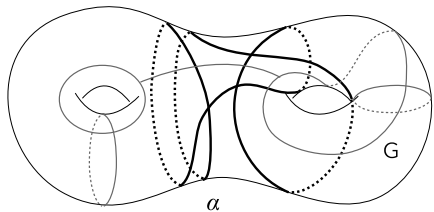}}
\caption{A filling graph $G$ that intersects the curve $\alpha$ minimally.}
\label{graph}
\end{figure}

If $\alpha$ doesn't fill $\Sigma$, for each region $R_k$ that is not a disc or a boundary annulus, take curves $a_{k,i}$ in $R_k$ parallel to the components of $\partial R_k$ and call the sub-surface of $R_k$ bounded by these curves $S_k$.
Now $\alpha$ fills $\Sigma'=\Sigma - \cup S_i$, and by the previous argument there is a graph $G'$ in $\Sigma'$ that cuts $\Sigma - \cup S_i$ into a polygonal disc $P'$ and crosses $\alpha$ at most $s+1$ times.
Observe that there exist arcs $b_{k,j}$ in $S_k$ with endpoints in $G'$ so that $\cup b_{k,i}$ cuts $S_k$ into a disc.
So $G=G' \cup a_{k,i} \cup b_{k,j}$ cuts $\Sigma$ into discs and $G$ crosses $\alpha$ at most $n+1$ times.
\end{proof}

\ 

Example. If an immersed curve $\alpha$ does not disconnect the surface $\Sigma$ then $n(\alpha)=i(\alpha)+1$. 
This is true because each simple curve in the image of $\alpha$ is an essential curve on $\Sigma$ (as there is a curve that intersect it in one point). So a graph $G$ that cuts $\Sigma$ into a disc must cut the image of $\alpha$ into trees, so $G$ must cross at least $i(\alpha)+1$ of the $2i(\alpha)$ arcs of $\alpha$ determined by its $i(\alpha)$ intersection points.

\

Now let's consider the relation between $n(\alpha)$ and $l(\alpha)$. By definition $n(\alpha) \leq l(\alpha)$. 
We claim that if $\Sigma -\alpha$ is connected then $n(\alpha)=l(\alpha)$. To show this, observe that if a graph $G$ cuts $\Sigma$ into a disc and $G$ has two or more vertices then we can connect two of them by a simple arc $e$ in $\Sigma-\alpha$, so $G \cup e$ cuts $\Sigma$ into two discs, and $e$ can be collapsed in $\Sigma$ to obtain a graph $G'$ that has less vertices and cuts $\Sigma$ into two discs, so we can erase one edge of $G'$ to obtain a graph $G''$ with less vertices than $G$ that cuts $\Sigma$ into a disc and $|\alpha \cap G''| \leq |\alpha \cap G|$. 
We can repeat the argument to obtain a graph with only one vertex, or a graph with all its vertices in $\partial \Sigma$, so $l(\alpha) \leq n(\alpha)$.

\

In general $l(\alpha)$ is larger than $n(\alpha)$ and it is not bounded in terms of $n(\alpha)$ or $i(\alpha)$ alone, but depends on the surface $\Sigma$.
We now define another complexity measure for $\alpha$ which will be used to bound $l(\alpha)$.

\bigskip

If $G$ is a graph on a surface $\Sigma$ and $x$ is a point in $\Sigma-G$, define the \textit{peeling} of $G$ from $x$ as the graph $\varphi_xG$ obtained from $G$ by erasing the edges that can be reached from $x$ without crossing $G$, and erasing the vertices of $G$ that lose all their edges.
Observe that if $\Sigma -G$ is connected then $\varphi_xG=\emptyset$ for every $x \in \Sigma -G$ (but the converse is not true). If we apply $\varphi_x$ repeatedly, then the number of components of the complement is reduced each time, so $\varphi^{|\Sigma-G|}_x G = \emptyset$ for every $x \in \Sigma-G$ and there is an $x$ for which $\varphi^{\lceil |\Sigma-G|/2 \rceil}_x G = \emptyset$.

\begin{figure}[h!]
\centerline{\includegraphics[scale=0.6]{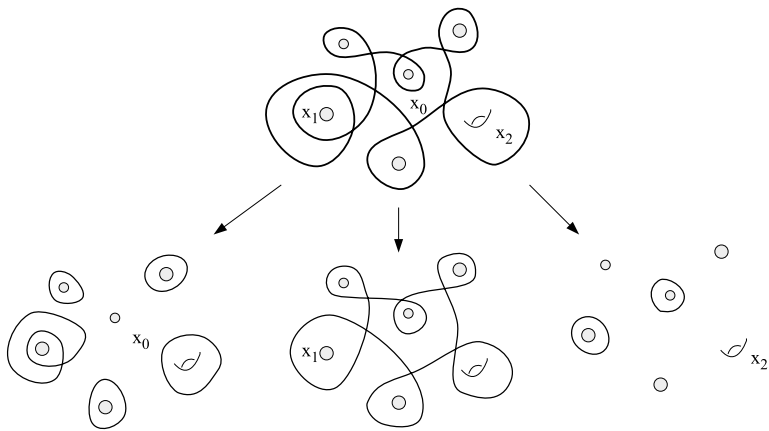}}
\caption{Peeling a curve from three different points.}
\end{figure}

We are interested in the case when $G$ is the image of an immersed curve $\alpha$, the definition above is necessary because applying $\varphi_x$ may transform the image of $\alpha$ into another type of graph.

Let $p_m(\alpha)$ be the minimum number of times that $\varphi_x$ has to be applied to erase all of $\alpha$, when $x$ is placed on the different components of $\Sigma-\alpha$, and, when $\Sigma$ has boundary, let $p_o(\alpha)$ be the minimum when $x$ is placed on the components that meet the boundary.

\begin{lemma} 
\label{pn}
If $\alpha$ is an immersed curve with minimal self-intersections in $\Sigma$ then $p_m(\alpha) \leq p_o(\alpha)  \leq n(\alpha)$, and there is a curve $\alpha'$ freely homotopic to $\alpha$ so that $p_m(\alpha') \leq \frac{1}{2}n(\alpha)+1$.
\end{lemma}

\begin{proof}
Let $G$ be a graph on $\Sigma$ that crosses $\alpha$ at $n$ points and cuts $\Sigma$ into a polygonal disc $D$, so
$\alpha$ intersects $D$ in $n$ arcs $a_1,a_2,...,a_n$. Since $\alpha$ has minimal self-intersections, these arcs are embedded and intersect each other at most once, so any pair of points in $D$ can be connected by a path that crosses each $a_i$ at most once. So $\varphi^n_x(\cup a_i)=\emptyset$ for every $x$ in $D-\cup a_i$, and therefore $\varphi^n_x(\alpha)=\emptyset$ for every $x$ in $\Sigma - \alpha$, so the first inequalities hold. 

To prove the second part we will show that the arcs $a_1,a_2,...,a_n$ are homotopic (fixing their endpoints) to arcs $a'_1,a'_2,...,a'_n$ in $D$ such that $\varphi^{\lceil \frac{n+1}{2} \rceil}_x(\cup a'_i)=\emptyset$ for some $x$ in $D-\cup a'_i$, and therefore $\varphi^{\frac{1}{2}n+1}_x(\alpha')=\emptyset$ for $\alpha'=\cup a'_i$ and for that $x$ in $\Sigma-\alpha'$.

\begin{figure}[h!]
\centerline{\includegraphics[scale=0.6]{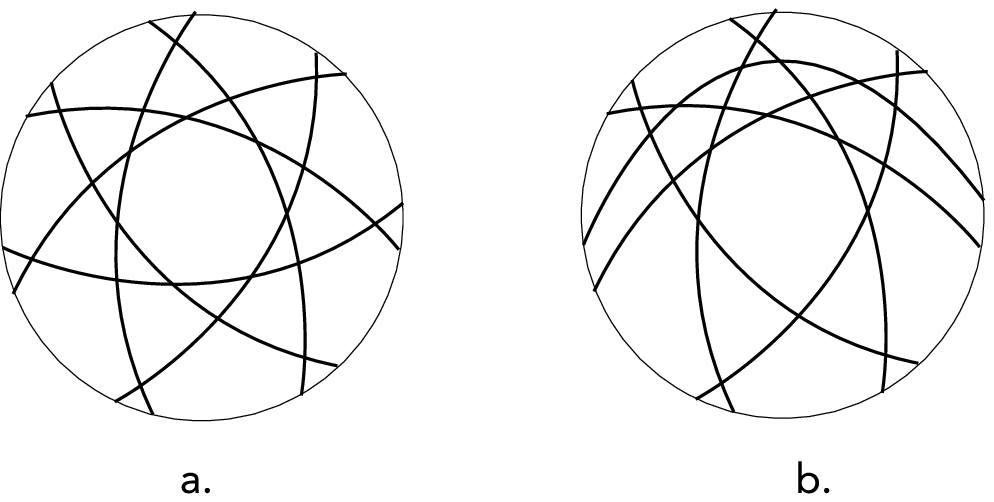}}
\caption{}
\label{arcconf}
\end{figure}

First consider the symmetric configurations of $n$ arcs obtained by rotating one arc (if $n$ is odd) or two arcs (if $n$ is even) so they all meet forming a star as in figure \ref{arcconf}a. If we place the point $x$ in the centre of the star then we can reach every other point in the disc by crossing at most $\lceil \frac{n+1}{2} \rceil$ of the arcs.

Not all `complete' configurations of $n$ arcs that meet each other have a point $x$ like that (see figure \ref{arcconf}b). However, since the endpoints of the arcs in all complete configurations have the same order in $\partial D$, there is a homotopy of the arcs that fixes their endpoints and takes them to a configuration isotopic to a star.
For the more general case of $n$ arcs that may not all meet each other, observe that a star configuration can be cut and pasted to get a configuration where the order of the endpoints of the arcs in $\partial D$ is prescribed arbitrarily. So any configuration of $n$ arcs is homotopic to a configuration obtained by cutting and pasting a star configuration, and the result follows because cutting and pasting does not increase the number of steps needed for peeling.
\end{proof}

\begin{theorem}{} 
\label{lp}
If $\alpha$ is a curve on a compact orientable surface $\Sigma$, then:
\begin{enumerate} 
\item $l(\alpha) \leq 2n(\alpha)-2$, if $\Sigma$ is planar and $\alpha$ is not simple.
\item $l(\alpha)\leq 2p_m(\alpha)(2-\chi(\Sigma))$, if $\Sigma$ is closed.
\item $l(\alpha)\leq 2p_o(\alpha)(1-\chi(\Sigma))$, if $\Sigma$ has boundary.
\end{enumerate}
\end{theorem}

\begin{proof}
1. Take a graph on $\Sigma$ that cuts the surface into a disc and crosses $\alpha$ at $n(\alpha)$ points. Since $\Sigma$ is planar the graph consists of one or more trees $T_1,T_2,...,T_k$. The boundary of a regular neighbourhood of $\cup T_i$ is a collection of properly embedded arcs $a_1,a_2,...$ that connect the boundary curves of $\Sigma$ and cut it into discs, so removing any $a_i$ from each $T_i$ leaves a subcollection of arcs that cut $\Sigma$ into a disc and cross $\alpha$ at most $2n$ times.
If some $a_i$ crosses $\alpha$ twice, we can remove that $a_i$ to obtain $l(\alpha) \leq 2n(\alpha)-2$. The same is true if at least two $T_i$'s are crossed by $\alpha$, since we can remove one $a_i$ that crosses $\alpha$ from each component.
We are left with the case when only one $T_i$ crosses $\alpha$ and each $a_i$ crosses $\alpha$ at most once.
Let $e_i,e_2,...e_n$ be the edges of $T_i$ that cross $\alpha$, and $T_{i1},T_{i2},...T_{in+1}$ be the subtrees obtained from $T_i$ by removing the edges $e_i,...,e_n$.
For each $T_{ij}$ consider a regular neighbourhood $N(T_{ij})$ contained in $N(T_i)$.
Take the collection formed by the arcs of $\partial N(T_j)$ for $j \neq i$, the arcs of $\partial N(T_{ij})$ for all the $T_{ij}$'s, and for each edge $e_i$, one arc $a_i$ of $\partial N(T)$ containing $e_i$.
Since this collection of properly embedded arcs cuts $\Sigma$ into discs and intersects $\alpha$ at most $n$ times, in this case $l(\alpha) \leq n(\alpha)$.

\

2. Take a point $x \in \Sigma - \alpha$ for which $\varphi_x^{p_m(\alpha)}\alpha=\emptyset$. Then $\alpha$ cuts $\Sigma$ into a finite collection of compact regions $R_0, R_1, R_2, ...$ where:
\begin{itemize}
    \item $R_0$ is the region that contains $x$.
    \item $R_{1},...,R_{i_1}$ are the regions adjacent (along arcs) to $R_0$ 
    \item $R_{i_1+1},...,R_{i_2}$ are the regions adjacent to $R_{1},...,R_{i_1}$.
    \item $R_{i_j+1},...,R_{i_{j+1}}$ are the regions adjacent to $R_{i_{j-1}+1},...,R_{i_j}$.
\end{itemize}
 
Each $R_i$ is a disc or can be cut along arcs that end at the self-intersection points of $\alpha$ to get a polygonal disc $P_i$. 
Now glue the $P_i$'s to get a single polygonal disc $P$ as follows:
Start with $P_0$, and for each $i$ with $i_j < i \leq i_{j+1}$ choose an arc $a_i$ of $P_i \cap P_{i'}$ for some $i' \leq i_{j}$ and attach $P_i$ to $P_{i'}$ along $a_i$.
$P$ is a fundamental region for $\Sigma$ and the universal cover $\tilde{\Sigma}$ is tessellated by the preimages $P_i$ of $P$.

By construction, for each side $s_j$ of $P$ there is a path $c_j$ in $P$ going from $x$ to the middle point of $s_j$ crossing $\alpha$ at most $p(\alpha)$ times. We can assume that the paths $c_1, ..., c_n$ are simple and intersect only at $x$, because if $c_i$ and $c_j$ intersect at another point $y$ we can switch their subpaths from $x$ to $y$ to get paths with less intersections.

Since $\cup c_j$ cuts $P$ into quadrilaterals, the preimage of $\cup c_j$ cuts the universal covering $\tilde{\Sigma}$ into discs $\Delta_k$ made of quadrilaterals surrounding a single vertex of the $P_i$'s, and the vertices of the $\Delta_k$'s are translates of $x$.
The deck transformations map $\Delta_k$'s to $\Delta_k$'s; and only the identity maps $\Delta_k$ to itself. 
Therefore some finite union of adjacent $\Delta_k$'s forms a fundamental region $F$ for $\Sigma$. As all the vertices of $F$ project to the same point in $\Sigma$, $F$ has $2(2-\chi(\Sigma))$ sides.

The length of $\alpha$ is one half of the number of intersections of $\partial{F}$ with the preimages of $\alpha$. Each side of $F$ is the union of two $c_j$'s, so it intersects the preimages of $\alpha$ at most $2p_m(\alpha)$ times, and $F$ has $2(2-\chi(\Sigma))$ sides, so $l(\alpha) \leq 2p_m(\alpha)(2-\chi(\Sigma))$.

3. The proof for surfaces with boundary is similar to the proof for closed surfaces, except that we need to start with a point $x$ on the boundary of $\Sigma$, so that at the end we can push all the $c_j$'s to disjoint arcs that end in $\partial \Sigma$. Thus the bound that we get is $l(\alpha) \leq 2p_o(\alpha)(1-\chi(\Sigma))$.
\end{proof}

The following examples show that the bounds given in theorem \ref{lp} are optimal: 

\medskip

Example. Take a curve $\alpha$ in an $n$-punctured sphere that zig-zags around the punctures as in fig \ref{zigzag}.
Then $l(\alpha)= 2 n(\alpha)-2$.

\begin{figure}[h!]
\centerline{\includegraphics[scale=0.7]{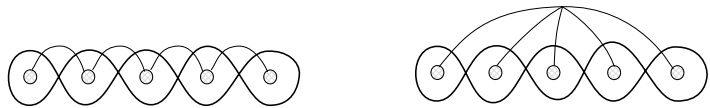}}
\caption{A planar curve $\alpha$ with $l(\alpha)= 2 n(\alpha)-2$}
\label{zigzag}
\end{figure}

Example. Consider a surface $\Sigma$ of genus $g$ with a single boundary component, and take $\alpha$ to be the n-th power of a loop homotopic to $\partial \Sigma$. If $\Sigma$ is cut along properly embedded arcs to get a disc, then each of the $2g$ arcs meets the boundary in two points, so it must cross $\alpha$ at least $2n$ times and so $l(\alpha)=2n \cdot 2g=2p_o(\alpha)(1-\chi(\Sigma))$.

\medskip

Example.
If $\Sigma$ is a closed surface of even genus $g$ and $\alpha$ is the n-th power of a simple loop $c$ that splits $\Sigma$ into two surfaces of genus $g/2$ (as in figure \ref{toro}) then $l(\alpha) = 2n \cdot 2 (g/2) = 
n \cdot 2g = 2p_m(\alpha)(2-\chi(\Sigma))$

\begin{figure}[h!]
\centerline{\includegraphics[scale=0.4]{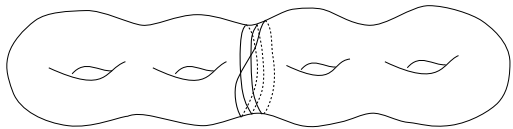}}
\caption{A curve $\alpha$ with $l(\alpha)=(i(\alpha)+1)(2-\chi(\Sigma))=n(\alpha) \cdot 2g(\Sigma)$}
\label{toro}
\end{figure}

\

If $\alpha$ is an immersed curve on $\Sigma$, define $m(\alpha)$ to be the minimum number of intersections between $\alpha$ and an essential curve on $\Sigma$. This minimum corresponds to the intersections of $\alpha$ with some simple curve, and is zero if $\alpha$ does not cut $\Sigma$ into discs. It is easy to see that $m(\alpha) \leq n(\alpha)$ and that $m(\alpha) \leq 2p_m(\alpha)$. 
One might expect $m(\alpha)$ to be much smaller than $n(\alpha)$, but the examples below show that $m(\alpha)$ can be as large as $\frac{1}{2} n(\alpha)$.
We now want to bound $m(\alpha)$ in terms of $i(\alpha)$.

\begin{lemma}{} 
\label{mn}
If $\alpha$ is a curve on a surface $\Sigma$, then $m(\alpha) \leq 2\sqrt{i(\alpha)}+1$.
\end{lemma}

\begin{proof}{}
We can assume that $\alpha$ cuts  $\Sigma$ into a collection of discs, since otherwise $m(\alpha)=0$.
We will show that if all the essential curves in $\Sigma$ intersect $\alpha$ at least $m$ times then $\Sigma - \alpha$ has at least $\frac{1}{4}(m-1)^2+2$ connected components, therefore
$i(\alpha)+2 \geq |\Sigma-\alpha| \geq \frac{1}{4}(m-1)^2+2$ and so $2\sqrt{i(\alpha)}+1 \geq m$.

The universal covering $\widetilde \Sigma$ is tessellated by the disc regions cut by the lines above $\alpha$.
Define the distance between two regions in $\widetilde \Sigma$ as the number of times that one has to cross the lines above $\alpha$ to go from one region to the other.
Since any essential curve on $\Sigma$ crosses $\alpha$ at least $m$ times, the regions at distance at most $\frac{m-1}{2}$ from $R$ project to different regions in $\Sigma$. 
The regions at distance $t$ from a region $R$ touch each other only along vertices, and for $t>1$,
the union of the regions at distance $t$ and $t-1$ from $R$ is an annulus. We now count the number of regions in these annuli when $R$ is crossed by a line $l$ that projects to a curve $c$ in $\Sigma$ that crosses $\alpha$ $m$ times.

The line $l$ crosses two regions $R'$ and $R''$ at distance $t$ from $R$. $R'$ and $R''$ are connected by two chains of regions that go around $R$ alternating between distance $t$ and $t-1$ from $R$. If $t \leq \frac{m-1}{2}$, each of these chains must contain at least $2t+1$ regions, as otherwise we could replace an arc of $c$ by a homotopic arc that crosses less regions, and obtain a curve that crosses $\alpha$ less than $m$ times. Since the two chains have only $R'$ and $R''$ in common, together they contain at least $4t$ distinct regions, and since they alternate between distance $t$ and distance $t-1$ from $R$, the number of regions at distance $t$ from $R$ is at least $2t$. 
Now the function $r(t)$ that counts the number of regions at distance at most $t$ from $R$ satisfies 
$r(t)-r(t-1) \geq 2t$ and $r(1) \geq 3$ so $r(t) \geq t^2+2$ for $t \leq \frac{m-1}{2}$.
So, if $m$ is odd, $r(\frac{m-1}{2}) \geq (\frac{m-1}{2})^2+2$ and we are done.

\begin{figure}[h!]
\centerline{\includegraphics[scale=0.6]{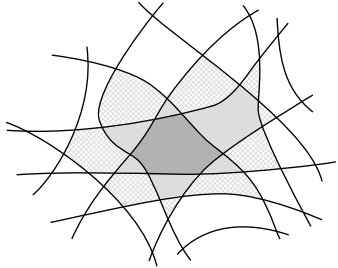}}
\caption{Regions at distance 1 and 2 from R.}
\end{figure}

If $m$ is even, the regions at distance $t=m/2$ from $R$ may not project to different regions of $\Sigma-\alpha$, but we claim that at least $m/2$ of those regions do so, and therefore $|\Sigma-\alpha | \geq r(\frac{m-2}{2})+\frac{m}{2} = (\frac{m-2}{2})^2+2 +\frac{m}{2} > (\frac{m-1}{2})^2 +2$.
To prove the claim, observe that the two regions crossed by $l$ at distance $m/2$ from $R$ project to the same region of $\Sigma-\alpha$, and they are joined by 2 chains of regions at alternating distances $m/2$ and $m/2-1$ from $R$. 
Since no essential curve on $\Sigma$ intersects $\alpha$ less than $m$ times,  these chains have length at least $m+1$, and the regions of any subchain of length $n$ must project to different regions of $\Sigma - \alpha$, and half of the regions are at distance $m/2$ from $R$.
\end{proof}

\medskip

Example. 
Figure \ref{myi}a. shows a filling curve $\alpha$ with $m(\alpha)=1$ and $i(\alpha)=3$, so $n(\alpha)=4$.
Figure \ref{myi}b. shows a curve $\gamma$ with $m(\gamma)=2$ and $i(\gamma)=3$, so $n(\gamma)\leq 4$. Both $\alpha$ and $\gamma$ have minimal self-intersections.
The powers of $\gamma$ satisfy $m(\gamma^n)=2n$, $n(\gamma) \leq 4n$ and $i(\gamma^n)=3n^2+n-1$ so
$m(\gamma^n) \geq \frac{1}{2} n(\gamma^n)$
and $m(\gamma^n) \sim \frac{\sqrt{3}}{2} \sqrt{i(\gamma^n)}$.

\begin{figure}[h!]
\centerline{\includegraphics[scale=0.8]{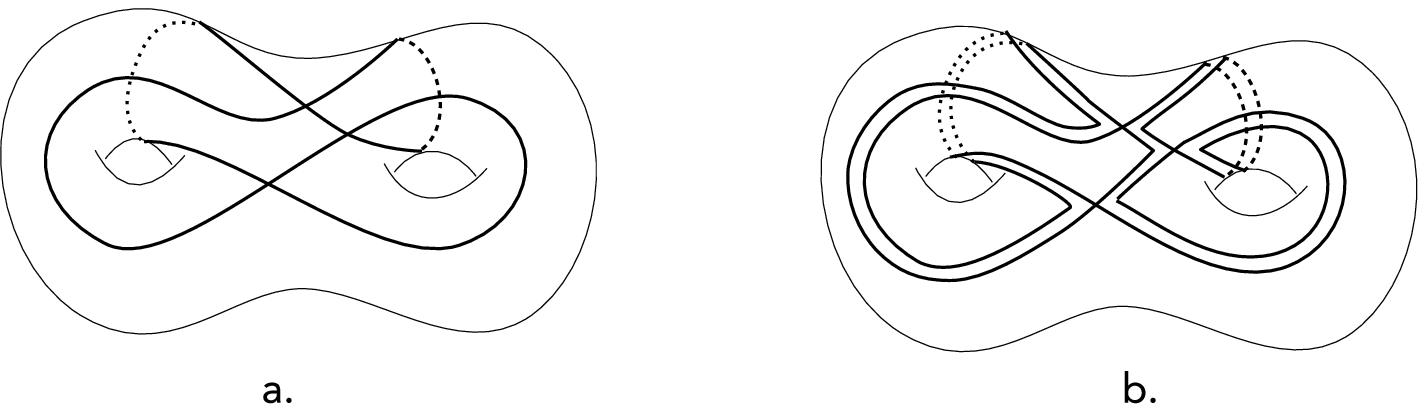}}
\caption{}
\label{myi}
\end{figure}

\section{Bounds regarding covering spaces}

Now we want to find estimates for the minimum degree $d(\alpha)$ of a covering to which a curve $\alpha$ lifts as an embedded curve, and the minimum degree $r(\alpha)$ to which $\alpha$ does not lift as a closed curve.

\begin{lemma} 
\label{n}
If a curve $\alpha$ on a surface $\Sigma$ lifts to an embedding on a covering $\widetilde \Sigma$
and $\alpha$ is freely homotopic to a curve $\alpha'$ by a homotopy that does not add intersections, then $\alpha'$ also lifts to an embedding on $\widetilde \Sigma$.
\end{lemma}

\begin{proof}
The free homotopy from $\alpha$ to $\alpha'$ lifts to a free homotopy from the embedded curve $\widetilde \alpha$ to a curve $\widetilde \alpha'$ in $\widetilde \Sigma$. If $\widetilde \alpha'$ is not an embedding then there is a subloop $\alpha'_1$ of $\alpha'$ that lifts to a subloop of $\widetilde \alpha'$. As the homotopy from $\alpha'$ to $\alpha$ does not remove any intersections, $\alpha'_1$ is freely homotopic to a subloop $\alpha_1$ of $\alpha$. As $\alpha'_1$ lifts to a closed loop in $\widetilde \Sigma$ then $\alpha_1$ also lifts to a closed loop, so $\widetilde \alpha$ is not an embedding, a contradiction.
\end{proof}

\

The following result from \cite{S} shows that if $\alpha$ is a curve with minimal self-intersections on a surface with boundary then $d(\alpha) \leq l(\alpha)$.

\begin{lemma}
\label{s}
Let $\Sigma$ be a surface with boundary and let $\mu_1,\mu_2,...,\mu_k$ be disjoint embedded arcs that cut $\Sigma$ into a disc $D$. If $\alpha$ is a curve or an arc with minimal self-intersections in $\Sigma$, then there is a covering of $\Sigma$ of degree at most $|\alpha \cap \cup \mu_i|$ (if $\alpha$ is a curve) or at most $|\alpha \cap \cup \mu_i| + 1$ (if $\alpha$ is an arc) to which $\alpha$ lifts as an embedding. 
\end{lemma}

\begin{proof}{} If $\alpha$ is a curve then it crosses $D$ along $|\alpha \cap \cup \mu_i|$ subarcs, and if $\alpha$ is an arc then it crosses $D$ along $|\alpha \cap \cup \mu_i|+1$ subarcs.
As $\alpha$ has minimal self-intersections, these subarcs $\alpha_1, \alpha_2,...$ are simple and intersect each other in at most 1 point. 
Let $\mu_i^+$ and $\mu_i^-$ be the two arcs of $\partial D$ that are identified into each $\mu_i$. Following $\alpha$,
let $\mu_{i_1}^{\sigma_1},\mu_{i_2}^{\sigma_2},...,\mu_{i_l}^{\sigma_l}$ be the signed sequence of intersections with the $\mu_i$'s, where $\sigma_j=+1$ if $\alpha_j$ ends at $\mu_{i_j}^+$ and $\sigma_i=-1$ if $\alpha_j$ ends at $\mu_{i_j}^-$.

If some $\alpha_j$ starts and ends at the same $\mu_i^{\pm}$ then there is a homotopy that pushes $\alpha_i$ into $\mu_i^{\pm}$ without increasing the intersections with other subarcs, so pushing it further across $\mu_i^{\pm}$ gives a homotopy of $\alpha$ that reduces its intersections with the $\mu_i$'s without adding self-intersections.
By Lemma \ref{n}, if the modified $\alpha$ lifts as an embedding in some covering of $\Sigma$ then the original $\alpha$ also lifts as an embedding.

To construct the covering $\widetilde \Sigma$, glue copies $D_1,D_2,...$ of the disc $D$, following $\alpha$ and going up to a new $D_j$ when $\alpha$ crosses $\mu_{i_j}$, identifying the arc $\mu_{i_j,j}^+$ in $\partial D_j$ with the arc $\mu_{i_j,j+1}^-$ in $\partial D_{j+1}$ if $\sigma_j=+1$, or the arc $\mu_{i_j,j}^-$ in $\partial D_j$ with the arc $\mu_{i_j,j+1}^+$ in $\partial D_{j+1}$ if $\sigma_j=-1$. If $\alpha$ is a curve do this modulo $|\alpha \cap \cup \mu_i|$, so the first and last copies of $D$ are also glued along an arc.

Now we can identify the remaining arcs $\mu_{i,j}^+$ with arcs $\mu_{i,j'}^-$, pairing them in any way, to obtain a covering $\widetilde \Sigma$ of $\Sigma$, the union of the arcs $\alpha_i$ in $D_i$ is a simple arc or curve on $\widetilde \Sigma$ that projects to $\alpha$.
\end{proof}

\begin{remark}
In the proof above it is not necessary to go up to a new $D_j$ every time that $\alpha$ crosses an arc $\mu_i$. 
If a simple subarc of $\alpha$ crosses the $\mu_i$'s in a sequence $\mu_{i_s}^{\sigma_s},...,\mu_{i_t}^{\sigma_s}$,  
where $\mu_{i_{s+1}},...,\mu_{i_{t-1}}$ are different from $\mu_{i_s}$ and $\mu_{i_t}$, and 
$\mu_{i_s}^{\sigma_s} \neq \mu_{i_t}^{-\sigma_t}$, then one can remain at the same $D_j$ at the intermediate $\mu_{i_j}$'s and go up only at $\mu_{i_s}$ and $\mu_{i_t}$ and we still get a simple arc or curve on $\widetilde \Sigma$ that projects to $\alpha$. 
\end{remark}

\medskip

Example. Take a $4n$-gon $P$ and $n$ arcs $\alpha_1,\alpha_2,...,\alpha_n$ in $P$ joining opposite even sides. Identify the even sides of $P$ in pairs to obtain a surface with boundary $\Sigma$, and choose the pairings so that the $\alpha_i$'s merge into a single closed curve $\alpha$. 
$P$ is a fundamental region for $\Sigma$, so every covering can be obtained by gluing copies of $P$. Since all the $\alpha_i$'s intersect each other, to undo the self-intersections each $\alpha_i$ must lift to a different copy of $P$, therefore for this example $d(\alpha)\geq n \geq l(\alpha)$.

\medskip

\begin{proposition}{} 
\label{dn}
If $\alpha$ is a curve with minimal self-intersections on a planar surface $\Sigma$, then
$d(\alpha) \leq n(\alpha)$.
\end{proposition}

\begin{proof}
If $\Sigma$ is planar, then the graph on $\Sigma$ that cuts it into a disc and crosses $\alpha$ at $n(\alpha)$ points is a collection of trees $T_1,T_2,...T_k$ that connect the boundary curves of $\Sigma$.
We want to replace these trees by a collection of disjoint, properly embedded arcs $\lambda_1,\lambda_2,...,\lambda_r$ that cut $\Sigma$ into a surface $\Sigma'$ and a collection of disjoint, properly embedded arcs $\mu_1, \mu_2,...,\mu_s$ in $\Sigma'$ that cut $\Sigma'$ into a disc.
We need the following.

\medskip

\noindent \textit{Claim.}
Every tree is a union of vertex-disjoint paths $\lambda_1,\lambda_2,...,\lambda_r$ and edge-disjoint paths  $\mu_1, \mu_2,...,\mu_s$ where each $\lambda_i$ joins vertices of degree 1 and each $\mu_k$ joins some $\lambda_i$ to another $\lambda_j$ or a vertex of degree 1.

To prove the claim, it is enough to show that in each tree $T$ there exist vertex-disjoint paths $\lambda_1,\lambda_2,...,\lambda_r$ that start and end at the vertices of degree 1 and contain all the vertices of degree larger than 2.
We will proceed by induction on the number of vertices of $T$, showing also that one can choose a vertex of degree 1 arbitrarily to belong or not to the paths.
For $v \leq 3$ the result is immediate. Now assume the result holds for all trees with less than $v$ vertices. Take a tree $T$ with $v$ vertices and cut it at a vertex $p$ of degree $g > 1$ to obtain trees $T_1, ..., T_g$ so that $p$ has degree 1 in each $T_i$.
Since each $T_i$ has less vertices than $T$, by the inductive hypothesis we can choose paths in $T_1, T_2$ and $T_3,...,T_g$ so that $p$ belongs to the paths in $T_1,T_2$ but not to the paths in the rest of the trees. 
Then, going back to $T$, we can take the paths in the smaller trees, gluing the two paths that contain $p$, to obtain disjoint paths with the required conditions.

By the previous claim, each tree $T_k$ is the union of paths $\lambda_i$ joining vertices of degree 1 and paths $\mu_j$ that join $\lambda_i$'s with other $\lambda_i$'s or with vertices of degree 1.
Let $\Sigma'$ be the surface obtained by cutting $\Sigma$ along the $\lambda_i$'s and let $\alpha_1,\alpha_2,...\alpha_k$ be the arcs of $\alpha$ in $\Sigma'$.
By lemma \ref{s}, for each $\alpha_k$ there is a covering $\Sigma_k$ of $\Sigma'$ to which $\alpha_k$ lifts as an embedded arc $\widetilde \alpha_k$
and the degree of the covering is the number of times that $\alpha_k$ crosses $\cup \mu_i$ plus 1.
If we glue the surfaces $\Sigma_k$ pairing the arcs in their boundaries above the two copies of each $\lambda_i$ in $\Sigma'$ we always get a covering of $\Sigma$. To get a covering where $\alpha$ lifts as a closed embedded curve it is enough to glue the arc of $\partial \Sigma_k$ where $\widetilde \alpha_k$ ends with the arc of $\partial \Sigma_k$ where $\widetilde \alpha_{k+1}$ starts for each $k$, and then glue the remaining arcs with corresponding arcs in any way. The degree of this covering is the number of times that $\alpha$ crosses the $\mu_i$'s plus the number of $\alpha_k$'s, which is the number of times that $\alpha$ crosses the $\lambda_i$'s.
\end{proof}

Now we want to find a bound for $d(\alpha)$ in terms of $l(\alpha)$ for all closed surfaces and a bound in terms of $n(\alpha)$ for all surfaces.
For this we first need to transform the graph $G$ that crosses $\alpha$ at $n(\alpha)$ points into a hierarchy for $\Sigma$.

\begin{lemma}{}  
\label{h} 
If a graph $G$ cuts a compact surface $\Sigma$ into a disc, then there is a family of disjoint embedded curves $\lambda_1,\lambda_2,...,\lambda_r$ that cut $\Sigma$ into a surface $\Sigma'$ and a family of disjoint embedded arcs $\mu_1, \mu_2,...,\mu_s$ in $\Sigma'$ that connect different components of $\partial \Sigma'$ and cut $\Sigma'$ into a disc, so that if $\alpha$ is a curve on $\Sigma$ then $|\alpha \cap (\bigcup \lambda_i \cup \bigcup \mu_j)| \leq 2 |\alpha \cap G|$.
\end{lemma}{}

\begin{proof}
We can assume that the vertices of $G$ have degree at most 3, because we can split the vertices of larger degree without changing the complement of $G$ or its intersections with $\alpha$.

Suppose first that $\Sigma$ is closed.
Let $G'$ be a maximal subgraph of $G$, whose connected components have planar neighbourhoods and have no separating edges. So an arc in $N(G')$ that crosses $G'$ once connects different component of $\partial N(G')$.
Observe that $G$ cannot have any cycle $c$ disjoint from $G'$, since $N(c)$ would be planar and disjoint from $N(G')$, and $G' \cup c$ would satisfy the given conditions. It follows that each connected component of $G-G'$ is a tree, and each path $p$ in $G-G'$ connects different components of $\partial N(G')$, as otherwise $G'\cup p$ would have a planar neighbourhood and no separating edges. 

For each component $C_i$ of $G'$, remove a boundary curve from $N(C_i)$ and connect the remaining boundary curves by short arcs that cross $N(C_i)$. Then $N(C_i)$ deformation-retracts to the union of these curves and arcs, which we call $\lambda_{i_1}, \lambda_{i_2},...$ and $\nu_{i_1},\nu_{i_2},...$ respectively. Let $\Sigma'$ be the surface obtained by cutting $\Sigma$ along the $\lambda_{i_j}$'s for all components of $G'$. 

As each component $T_j$ of $G-G'$ in $\Sigma'$ is a tree, and each path of $T_j$ connects different components of $\partial \Sigma'$, then the arcs of $\partial N(T_i) -\partial \Sigma'$ connect different components of $\partial \Sigma'$. If we remove one of these arcs then $N(T_i)$ deformation-retracts to the union of the remaining arcs (which we call $\mu_{i_1},\mu_{i_2},...$).
So the union of all the $\nu_{i_j}$'s and the $\mu_{j_l}$'s for all the components of $G'$ and $G-G'$ cuts $\Sigma'$ into a disc. 

Now suppose that $\Sigma$ has boundary. Then we can deform $G$ to a graph that touches each boundary component at one point, without changing its complement or the number of intersections with $\alpha$. For, if two edges of $G$ end at the same boundary curve, then we can slide the end of one edge along the curve until it reaches the other edge (see figure \ref{wrap}). The same construction for the closed case works: $G'$ does not touch the boundary because it does not have separating edges, so the boundary of $\Sigma'$ consists of curves of $\partial \Sigma'$ and curves of $\partial N(G')$, and the trees that form $G-G'$ connect different boundary curves of $\Sigma'$ as $G$ only touches each component of $\partial \Sigma$ once.
\end{proof}

\

\begin{figure}[h!]
\centerline{\includegraphics[scale=0.7]{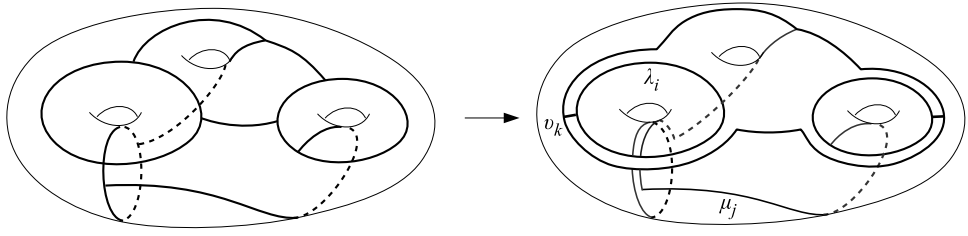}}
\caption{From the graph $G$ to a hierarchy for $\Sigma$.}
\label{grafica}
\end{figure}

\begin{lemma}
\label{p}
Let $\Sigma$ be an orientable surface and $\{\mu_1,\mu_2,...,\mu_k\}$ a collection of properly embedded arcs on $\Sigma$ that connect different components of $\partial \Sigma$ and cut $\Sigma$ into a disc.
If $\alpha$ is an arc on $\Sigma$ that has minimal self-intersections and cannot be homotoped into $\partial \Sigma$, then there is a covering $\widetilde \Sigma$ of $\Sigma$ of degree at most $|a \cap \cup \mu_i| + 2$, where $\alpha$ lifts as an embedded arc $\widetilde \alpha$ with the following property:

(*) $\widetilde \alpha$ connects two different components of $\partial \widetilde \Sigma$ that project with degree 1 to $\partial \Sigma$.
\end{lemma}

\begin{proof}
Let $c$ and $c'$ be the boundary curves of $\Sigma$ where $\alpha$ starts and ends ($c$ may be equal to $c'$).

If $\alpha$ does not cross any $\mu_i$'s then $\alpha$ is a simple arc. If $c \neq c'$ then $\alpha$ satisfies condition (*) and we can take $\widetilde\Sigma=\Sigma$.
If $c=c'$ and $\alpha$ does not separate $\Sigma$ then there is a curve that crosses $\alpha$ once, so pasting two copies of $\Sigma$ cut along that curve we get a double covering of $\Sigma$ to which $\alpha$ lifts as an embedding satisfying (*).
If $c=c'$ and $\alpha$ separates $\Sigma$ then each component of $\Sigma|_{\alpha}$ has at least two boundary curves (otherwise it would have positive genus and would have to be crossed by some $\mu_j$, but these arcs cannot have both endpoints in $c$ and so they would have to cross $\alpha$). Then there is an arc that crosses $\alpha$ at one point and connects boundary components of $\Sigma$ different from $c$, so pasting two copies of $\Sigma$ cut along that arc we get a double covering of $\Sigma$ to which $\alpha$ lifts as an embedding satisfying (*).

Now let's assume that $\alpha$ crosses some $\mu_i$'s.
Applying directly the construction in lemma \ref{s} may not work, because $\alpha$ may start by wrapping around $c$, or end by wrapping around $c'$ as in figure \ref{wrap}a, so going up or staying at the same level at every $\mu_i$ may force the boundary curve of $\widetilde \Sigma$ above $c$ to wrap around $c$ or $c'$ more than once.

\begin{figure}[h!]
\centerline{\includegraphics[scale=0.8]{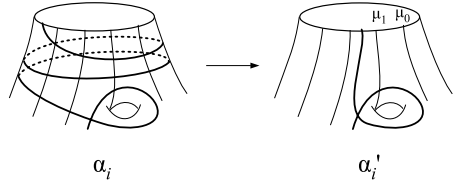}}
\caption{Unwrapping the beginning of $a$.}
\label{wrap}
\end{figure}

To avoid this, start by unwrapping the two ends of $\alpha$, doing a sequence of triangle moves that remove intersections of $\alpha$ with some $\mu_j$'s, as in figure \ref{wrap}b. These triangle moves do not add self-intersections to $\alpha$, and can only remove self-intersections when a triangle formed by the initial and final arcs of $\alpha$ and an arc of $\partial \Sigma$ is removed.
Observe that if there is a covering of $\Sigma$ where the modified arc $\alpha'$  lifts as an embedding $\widetilde \alpha'$ satisfying condition (*), then $\alpha'$ can be isotoped back to $\alpha$ and $\widetilde \alpha$ will be isotopic by triangle moves to $\widetilde \alpha'$. So $\widetilde \alpha$ satisfies condition (*) and $\widetilde \alpha$ is an embedding because a triangle move of $\alpha$ can create a self-intersection of $\widetilde \alpha$ 
only when the initial and final arcs of $\widetilde \alpha$ can form a triangle with an arc of $\partial \widetilde \Sigma$, which can only happen if $\widetilde \alpha$ starts and ends at the same boundary curve, contradicting condition (*).

Now let $\mu_{i_1}^{\sigma_1},\mu_{i_2}^{\sigma_2},...,\mu_{i_l}^{\sigma_r}$ be the signed sequence of crossings along $\alpha$.
We build the covering of $\Sigma$ as in lemma \ref{s}, taking $r+1$ copies of $\Sigma|_{\cup \mu_1}$ and gluing them along the $\mu_i$'s crossed by $\alpha$, identifying the arc $\mu_{i_j,j}^{\sigma_j}$ at level $j$ with the arc $\mu_{i_j,j+1}^{-\sigma_j}$ at level $j+1$. All the other arcs $\mu_{k,j} ^{\pm}$ are free, i.e., they are not identified (yet) with other arcs $\mu_{k,j'} ^{\mp}$. We need to show that these identifications can be made so that the boundary curves where the lifting of $\alpha$ starts and ends are different and project to $c$ and $c'$ with degree 1.

If $\mu_{i_1}$ does not touch $c$, then  the arcs $\mu_{j,1}^{\pm}$ are free for all the $\mu_j$'s that touch $c$, so if we identify all these $\mu_{j,1}^{-}$ with $\mu_{j,1}^{+}$ we get a curve $\tilde c$ that stays at level $1$ and covers $c$ once.

If $\mu_{i_1}$ touches $c$, the direction in which $\alpha$ crosses $\mu_{i_1}$ determines a direction for $c$. Label the $\mu_j$'s that touch $c$ cyclically in this direction as  $\mu_0,\mu_1,...$, where $\mu_{i_1}=\mu_1$, and assume that $c$ crosses them with positive signs. Observe that the initial point $p$ of $\alpha_i$ does not lie in the interval of $c$ between $\mu_0$ and $\mu_1$, otherwise we could unwind $\alpha$ and remove its intersection with $\mu_1$.

As $\mu_{i_1}^{\sigma_1} = \mu_1$, then $\mu_{1,1}^+$ is identified with $ \mu_{1,2}^-$ and all the other $\mu_{j,1}^{\pm}$ are free, so we can identify $\mu_{j,1}^{-}$ with $\mu_{j,1}^{+}$ for all $j \neq 0,1$. We want to identify $\mu_{1,1}^-$ with $\mu_{1,k}^+$ and $\mu_{0,1}^+$ with $\mu_{0,k}^-$ for some $k$, as in figure \ref{unwind}a, so that $\widetilde c$ stays at level 1 except for the interval between $\mu_0$ and $\mu_1$ where it goes up to level $k$.
Let $k$ be the first index for which $\mu_{i_k}^{\sigma_k} \neq \mu_1,\mu_{0}^{-1}$, if there is one. Then $\mu_{1,k}^+$ and $\mu_{0,k}^-$ are free (otherwise $\mu_{i_{k-1}}^{\sigma_{k-1}} = \mu_1^{-1},\mu_{0}$ or $\mu_{i_k}^{\sigma_k} = \mu_1,\mu_{0}^{-1}$). 
If $\mu_{i_j}^{\sigma_j} = \mu_1,\mu_{0}^{-1}$ for all $j$ then $\mu_{1,r+1}^+$ and $\mu_{0,r+1}^-$ are free, and we can take $k=r+1$.

\begin{figure}[h!]
\centerline{\includegraphics[scale=0.8]{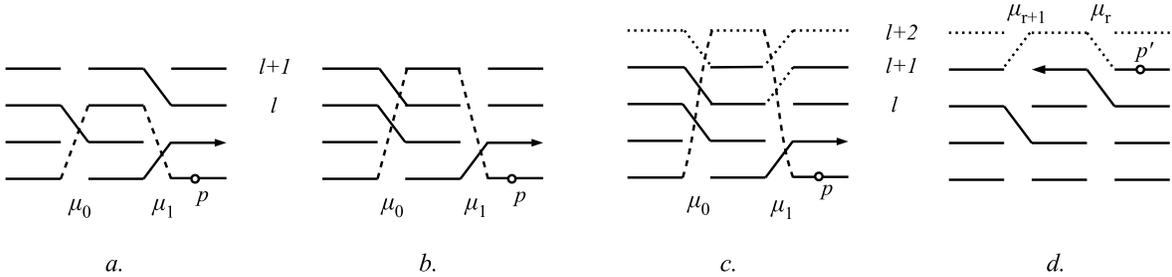}}
\caption{Gluings at the boundary curves above $c$ and $c'$.}
\label{unwind}
\end{figure}

Now we have to arrange that the curve $\tilde c'$ that contains the final point of $\widetilde a$ covers $c'$ once, and that $\tilde c \neq \tilde c'$. This may require adding another level to the covering if the index $k$ above is $r+1$.

If $\mu_{i_r}$ does not touch $c'$, then  the arcs $\mu_{j,r+1}^{\pm}$ are free for all the $\mu_j$'s that touch $c'$, so if we identify all these $\mu_{j,r+1}^{-}$ with $\mu_{j,r+1}^{+}$ we get a curve $\widetilde c'$ covers $c'$ once and stays at level $r+1$, therefore $\widetilde c \neq \widetilde c'$.

If $\mu_{i_r}$ touches $c'$, the direction in which $\alpha$ crosses $\mu_{i_r}$ determines a direction for $c'$. 
Label the $\mu_j$'s that touch $c'$ cyclically in that direction as $\mu_r,\mu_{r+1},...$, so that $\mu_{i_r}= \mu_r$, and assume that $c'$ crosses them with positive signs.
Observe that the final point $p'$ of $\widetilde \alpha$ does not lie in the interval of $c'$ between $\mu_r$ and $\mu_{r+1}$.

If $k \leq r$ then the only arc identified so far at the level $k+1$ is $\mu_{r,r+1}^-$ and all the other arcs $\mu_{j,r+1}^{\pm}$ are free, therefore we can identify $\mu_{j,r+1}^{+}$ with $\mu_{j,r+1}^{-}$ for all $j \neq r, r+1$ and identify $\mu_{r,r+1}^+$ with $\mu_{r,r+2}^-$ and $\mu_{r+1,r+1}^-$ with $\mu_{r+1,r+2}^+$. Now the curve $\widetilde c'$ that contains $p'$ stays at the level $r+1$ except for one arc at level $r+2$ and condition (*) holds.

If $k=r+1$ then $\mu_r = \mu_1$ or $\mu_r  = \mu_0^{-1}$ and the identifications
$\mu_{1,1}^-$ with $\mu_{1,r+1}^+$ and  $\mu_{0,1}^+$ with $\mu_{0,r+1}^-$ need to be changed to
$\mu_{1,1}^-$ with $\mu_{1,r+2}^+$ and  $\mu_{0,1}^+$ with $\mu_{0,r+2}^-$ (see figure \ref{unwind}c), so that the only arc identified at level $r+1$ is $\mu_{r,r+1}^-$, and the only arcs identified at level $r+2$ are 
$\mu_{1,r+2}^+$ and $\mu_{0,r+2}^-$. 
But one of these arcs is $\mu_{r,r+2}^+$ and so the other cannot be $\mu_{r+1,r+2}^+$ 
(otherwise $\{\mu_r,\mu_{r+1}\}=\{\mu_1,\mu_0^{-1}\}$,
but as $\Sigma'$ is orientable, if the signs of the intersections of $\mu_1$ and $\mu_0$ with $c$  agree, then the signs of the intersections of $\mu_1$ and $\mu_0$ with $c'$ must also agree).
Therefore we can identify $\mu_{r,r+1}^+$ with $\mu_{r,r+2}^-$ and $\mu_{r+1,r+1}^-$ with $\mu_{r+1,r+2}^+$
and identify all the arcs $\mu_{j,r+1}^{+}$ with $\mu_{j,r+1}^{-}$ for all $j \neq r, r+1$. The curve $\widetilde c'$ that contains $p'$ stays at the level $r+1$ except for one arc at level $r+2$ and condition (*) holds.
\end{proof}

\begin{remark}
The previous proof starts by gluing $|\alpha \cap \cup \mu_i|+1$ copies of the surface $\Sigma|_{\cup \mu_i}$ along the $\mu_i$'s crossed by $\alpha$, going up to a new copy at each of them. It may be possible to get an embedded arc $\widetilde \alpha$ going up or down or staying at the same copy when crossing the $\mu_i$'s by using less copies of  $\Sigma|_{\cup \mu_i}$, but it is not clear that one can complete the construction of a covering with condition (*) by adding just one more copy of $\Sigma|_{\cup \mu_i}$. However, if only one $\mu_i$ is glued at the first and the last levels then it can be done by adding 2 more copies of $\Sigma|_{\cup \mu_i}$:
Since only the arc $\mu_{1,1}+$ is identified at the level 1 and only the arc $\mu_{r,r+1}^-$ is identified at the level $s+1$, then we can identify $\mu_{1,1}^-$ with $\mu_{1,s+2}^+$ and $\mu_{0,1}^+$ with $\mu_{0,s+2}^-$.
and identify $\mu_{j,1}^{-}$ with $\mu_{j,1}^{+}$ for all $j \neq 0,1$ to make $\widetilde c$ cover $c$ once, 
and we can identify $\mu_{r,r+1}^-$ with $\mu_{s,s+3}^+$ and $\mu_{s+1,s+1}^+$ with $\mu_{s+1,s+3}^-$
and identify $\mu_{j,r+1}^{+}$ with $\mu_{j,r+1}^{-}$ for all $j \neq r, r+1$ to make $\widetilde c'$ cover $c'$ once, 
and by construction $\widetilde c$ and $\widetilde c'$ are different curves, so condition (*) is satisfied.

\end{remark}

\medskip

\begin{theorem}{} 
\label{dn2}
If $\alpha$ is a curve with minimal self-intersections on an orientable surface $\Sigma$, $d(\alpha) < 5n(\alpha)$.
\end{theorem}

\begin{proof}
Let $G$ be a graph on $\Sigma$ that cuts it into a disc and has $n(\alpha)$ intersections with $\alpha$. 
By lemma \ref{h} there is a family of disjoint embedded curves $\lambda_1,\lambda_2,...,\lambda_r$ that cut $\Sigma$ into a surface $\Sigma'$ and a family of disjoint embedded arcs $\mu_1, \mu_2,...,\mu_s$ that cut $\Sigma'$ into a disc, so that $|\alpha \cap (\bigcup \lambda_i \cup \bigcup \mu_j)| \leq 2 |\alpha \cap G|$.

If $\alpha$ does not cross the $\lambda_i$'s then $\alpha$ is a closed curve on $\Sigma'$
and we can find a covering $\widetilde \Sigma'$ of $\Sigma'$ to which $\alpha$ lifts as an embedded curve as in lemma \ref{s}. 
If for each $\lambda_j$, the two curves $\lambda_j'$ and $\lambda_j''$ in $\partial \Sigma'$ have the same number of preimages of each degree in $\widetilde \Sigma'$, then we can identify them to get a covering of $\Sigma$ to which $\alpha$ lifts as an embedding.
There is no reason to expect that this will always happen, but we can take an automorphism $\phi$ of $\Sigma'$ 
that interchanges each $\lambda_j'$ with $\lambda_j''$ and `double' the covering $p:\widetilde \Sigma' \to \Sigma'$ by taking a copy $\widetilde \Sigma_i''$ of $\widetilde \Sigma_i'$ and considering the covering $\phi \circ p:\Sigma_i'' \to \Sigma'$.
Now $\widetilde \Sigma' \cup \widetilde \Sigma_i''$ has the same number of preimages of each degree for each $\lambda_j'$ and $\lambda_j''$, and therefore we can glue them along their boundaries to get a covering of $\Sigma$ to which $\alpha$ lifts as a closed embedded curve.
The degree of this covering is at most $2 |\alpha \cap \bigcup \mu_j| = 4|\alpha_i \cap G| = 4n(\alpha)$.

\

If $\alpha$ crosses some $\lambda_i$'s then these split $\alpha$ into a collection of arcs $\alpha_1, \alpha_2,...$ in $\Sigma'$. 
By lemma \ref{p}, for each $\alpha_i$ that is not homotopic into the boundary of $\Sigma'$ there is a covering $\widetilde \Sigma_i'$ of $\Sigma'$ of degree at most $|\alpha_i \cap \cup \mu_j| + 2$ to which $\alpha_i$ lifts as an embedded arc $\widetilde \alpha_i$ satisfying condition (*).

If some $\alpha_i$ is homotopic to an arc in the boundary of $\Sigma'$, there are no coverings of $\Sigma'$ to which $\alpha_i$ lifts starting and ending at different boundary curves, so we need to modify the previous construction.

Consider a maximal sequence of arcs $\alpha_{s} \alpha_{s+1}...\alpha_{s'}$ of $\alpha$ that can be homotoped to some $\lambda_u$, then a subarc $a$ of $\alpha$ that starts just before $\alpha_{s}$ and ends just after $\alpha_{s}$ is homotopic, fixing its endpoints,  to an arc $a'$ in a neighbourhood of $\lambda_u$ that crosses $\lambda_u$ at most once.
As $\alpha$ has minimal self-intersections, the homotopy can be done by a sequence of bigon and triangle moves with $\lambda_u$ and the adjacent $\mu_j$'s, so the number of self-intersections of $\alpha$ does not change. The homotopy reduces the number of intersections of $\alpha$ with $\lambda_u$ by at least 2, but it can increase the number of intersections with the $\mu_j$'s even more, as can be seen in figure \ref{move}:

\begin{figure}[h!]
\centerline{\includegraphics[scale=0.6]{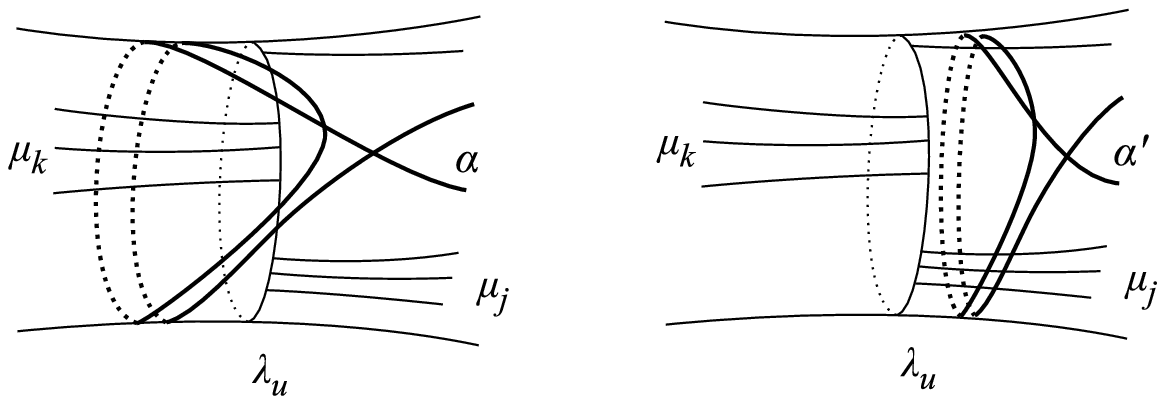}}
\caption{}
\label{move}
\end{figure}

Let $\alpha'$ be the image of $\alpha$ after the homotopy, then the homotopy reduces the sequence of arcs $\alpha_{s-1} \alpha_{s} \alpha_{s+1}...\alpha_{s'} \alpha_{s'+1}$ to one formed by two arcs $\alpha'_{s-1}\alpha'_{s'+1}$ or a single arc $\alpha'_{s}$ that are not homotopic into $\lambda_u$.
In the first case, although $\alpha'_{s-1}$ and $\alpha'_{s'+1}$ may intersect the $\mu_j$'s more times than 
$\alpha'_{s-1}$ and $\alpha'_{s'+1}$, the new intersections occur if the homotopy wraps the ends of these arcs around $\lambda_u$, but in the proof of lemma \ref{p} these intersections do not count towards the degree of the covering.

In the second case, the subarc $a$ of $\alpha$ is made of a loop representing some power $\lambda_u^e$ and two simple arcs, and the same is true for the homotopic arc $a'$ of $\alpha'$ that lies in a neighbourhood of $\lambda_u$. So $a$ crosses the union of the $\mu_j$'s and $\lambda_t$ at least $n$ times.
$a'$ may cross the union of $\mu_j$'s many more times, but we can assume that $a'$ crosses them in a sequence that follows the sequence of crossings of $\lambda_u$ with the $\mu_j$'s on that side of $\lambda_u$, crossing each of them at most $t+1$ times. 
By the remark after lemma \ref{s}, to lift $\alpha'_s$ as an embedding it is not necessary to go up at all the crossings of $a'$ with $\mu_j$'s, one can just go up at the first $\mu_j$ crossed by $a'$ and every time that $a''$ crosses it, and then go up at the last $\mu_j$ crossed by $a'$, so $\alpha'_s$ goes up at most $\alpha_{s} \alpha_{s+1}...\alpha_{s'} \cap \cup \mu_j +1$ times.
By the remark after lemma \ref{p} we can find a covering of $\Sigma'$ where $\alpha'_s$ lifts as an embedding satisfying condition (*) of degree at most 
$|\alpha_{s} \alpha_{s+1}...\alpha_{s'} \cap \cup \mu_j| +1+3 < |\alpha_{s-1} \cap \cup \mu_j|+2+|\alpha_{s} \cap \cup \mu_j|+2+...+ |\alpha_{s'+1} \cap \cup \mu_j|+2$.

Now, glue the coverings $\Sigma_i'$ corresponding to the arcs $\alpha_i$ along the boundary curves where the liftings of the arcs start and end, so $\alpha$ lifts to a closed embedded curve. The rest of the boundary curves of the $\Sigma_i'$ still need to be paired with other curves to obtain a covering of $\Sigma$. As mentioned before, this can be done if for each $\lambda_j$, the two curves $\lambda_j'$ and $\lambda_j''$ in $\partial \Sigma'$ have the same number of preimages of each degree in $ \sqcup \Sigma_i'$. If this does not happen, we can double the $\Sigma_i$'s with different number of preimages (as we did in the case when $\alpha$ didn't cross the $\lambda_i$'s) to guarantee that the pairing can be done.

To estimate the degree of the covering, observe that the arcs $\alpha_1,\alpha_2,...$ in $\Sigma'$ whose union is $\alpha$ can be divided in two classes: `long' $\alpha_i$'s corresponding to the arcs of $\alpha$ cut by the subgraph $G'$, and `short' $\alpha_i$'s  where $\alpha$ crosses $N(G')$. The long $\alpha_i$'s contain all the intersections of $\alpha$ with the $\mu_j$'s, and the short $\alpha_i$'s are simple and connect different boundary curves of $\Sigma'$.

The coverings $\Sigma_i$ for the short $\alpha_i$'s are trivial and there is no need to double them, the coverings
$\Sigma_i$ for the long $\alpha_i$'s have degree $|\alpha_i \cap \bigcup \mu_j|+2$ and may have to be doubled. If
$S$ is the set of indices corresponding to short $\alpha_i$'s and $L$ is the set of indices corresponding to long $\alpha_i$'s, then $|S| < |\alpha \cap G'|$ and $|L| = |\alpha \cap G'|$ and the total degree of the covering is at most

\medskip

$\sum\limits_{i \in S}1 + 2 \sum\limits_{i \in L} |\alpha_i \cap \bigcup \mu_j|+2 < |\alpha \cap G'|+2(2|\alpha \cap (G-G')|+2 |\alpha \cap G'|) \leq 5 |\alpha \cap G| =5n(\alpha)$.
\end{proof}

\medskip

The previous bound does not appear to be optimal, we conjecture that $d(\alpha) < 2n(\alpha)$.

\medskip

Example. If a curve $\alpha$ is formed by two subloops $x^n$ and $y$, where $x$ and $y$ are simple, disjoint and non-separating, then $n(\alpha)=n+1$ and $d(\alpha) \geq n+1$ so $d(\alpha) \geq n(\alpha)$. If $x$ is separating then $n(\alpha) = n+1$, and we believe, but cannot prove, that $d(\alpha)=2n-1$.

\medskip

Recall that $r(\alpha)$ is the minimum degree of a covering of $\Sigma$ to which $\alpha$ does not lift as a closed curve, so $r(\alpha)$ is the index of the largest subgroup of $\pi_1(\Sigma)$ that does not contain an element in the conjugacy class of $\alpha$.
One would expect (as it is often the case) that $r(\alpha) < d(\alpha)$, but the following example shows that this is not always true.

\medskip

Example. \textit{A curve $\alpha$ with $r(\alpha)>d(\alpha)$.} The curve $\alpha$ in figure \ref{dyr} represents the conjugacy class of $abab^{-1}$ in $\pi_1(\Sigma)$. As $\alpha$ is 0 in $H_1(\Sigma, \mathbf{Z_2})$ then $\alpha$ lifts to all the double coverings of $\Sigma$, therefore $r(\alpha)>2$. As $\alpha$ lifts as an embedding to the double covering obtained by cutting $\Sigma$ along the dotted curve, $d(\alpha)=2$.

\begin{figure}[h!]
\centerline{\includegraphics[scale=0.6]{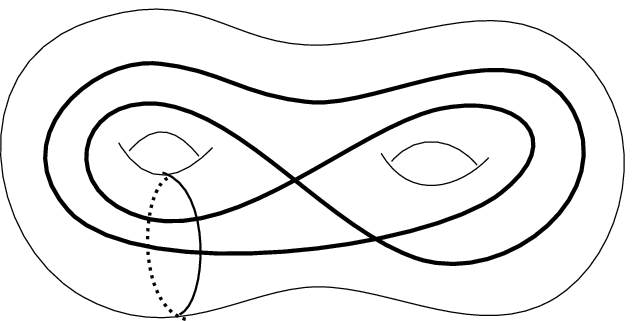}}
\caption{A curve $\alpha$ with $r(\alpha)>d(\alpha)$}
\label{dyr}
\end{figure}

\begin{corollary}
If $\alpha$ is an essential curve on a surface $\Sigma$ then 
\begin{enumerate}
\item $r(\alpha)=2$ or $3$ or $r(\alpha) < n(\alpha)$ \hskip 10pt if $[\alpha] \neq 0$ in $H_1(\Sigma)$.  
\item  $r(\alpha) \leq n(\alpha)+1$ \hskip 10pt if $\Sigma$ is planar.
\item  $r(\alpha) \leq 5n(\alpha)$ \hskip 10pt if $\Sigma$ is orientable.
\end{enumerate}
\end{corollary}

\begin{proof}
We can assume that $\alpha$ has minimal self-intersections because if a curve lifts to a covering of $\Sigma$ then all the free homotopies of the curve lift to that covering.

To prove (1) cut $\Sigma$ into a polygonal disc $D$ that is crossed by $\alpha$ along $n(\alpha)$ arcs. Then $H_1(\Sigma)$ is generated by curves that cross $D$ along one path connecting corresponding sides of $D$.
As $\alpha$ has minimal self-intersections, the arcs of $\alpha \cap D$ are simple and cross each other in at most 1 point, therefore each pair of corresponding sides of $D$ is connected by a path in $D$ that crosses each arc of $\alpha \cap D$ at most once. These paths give generators $h_1,h_2,...$ for $H_1(\Sigma)$ such that $|h_i \cap \alpha| \leq n(\alpha)$ for each $h_i$.
If $[\alpha] \neq 0$ in $H_1(\Sigma)$ then the algebraic intersection number $[\alpha] \cap [h_i]$ is not 0 for some $h_i$, and $[\alpha] \cap [h_i] \neq 0 (mod N)$ for the smallest $N$ that does not divide it, which is smaller than $[\alpha] \cap [h_i]$ unless this is $1$ (so $N=2$) or 2 (so $N=3$).

To prove (2) and (3) observe that in the constructions in proposition \ref{dn} and in theorem \ref{dn2}, instead of gluing the coverings of $\Sigma|_{\cup \lambda_i}$ where the arcs $\alpha_1, \alpha_2,...$ lift as embeddings satisfying condition (*) cyclically to close $\widetilde \alpha$, we can insert an extra copy of $\Sigma|_{\cup \lambda_i}$ at the end so that $\widetilde \alpha$ does nor close up and becomes an embedded arc.
\end{proof}

\end{document}